\documentclass[11pt,a4paper]{article}
\usepackage[utf8]{inputenc}
\usepackage{amsmath}
\usepackage{amsfonts}
\usepackage{amssymb}
\usepackage{amsthm}
\usepackage{makeidx}
\usepackage{graphicx}
\usepackage{authblk}
\usepackage{subfig}
\usepackage{subeqnarray}
\usepackage[T1]{fontenc}
\usepackage[top=2cm, bottom=2cm, left=2cm, right=2cm]{geometry}
\usepackage{marginnote}

\usepackage{imakeidx}
\usepackage{hyperref}
\usepackage{color}
\makeindex[intoc]

\usepackage{enumerate}

\newtheorem{theorem}{Theorem}[section]
\newtheorem{corollary}[theorem]{Corollary}

\newtheorem{proposition}[theorem]{Proposition}

\title{Concentrated reaction terms on the boundary of \\ rough domains for a quasilinear equation}
\date{}

\author[1]{Ariadne Nogueira\thanks{e-mail: ariadnen@ime.usp.br}}
\author[1]{Jean Carlos Nakasato\thanks{e-mail: nakasato@ime.usp.br}}
\author[1]{Marcone Corr\^ea Pereira\thanks{e-mail: marcone@ime.usp.br. Partially supported by FAPESP 2017/02630-2 and CNPq 303253/2017-7.}}
\affil[1]{Depto. Matem\'atica Aplicada, IME, Universidade de S\~ao Paulo,
	Rua do Mat\~ao 1010, S\~ao Paulo - SP, Brazil}

\begin{document}

	\maketitle
		\begin{abstract}
		In this work we analyze the solutions of a $p$-Laplacian equation with homogeneous Neumann boundary conditions set in a family of rough domains with a nonlinear term concentrated on the boundary. At the limit, we get a nonlinear boundary condition capturing the oscillatory geometry of the strip where the reactions take place. 
		\end{abstract}

	\noindent \emph{Keywords:} $p$-Laplacian, Neumann problem, Oscillating domains, Asymptotic analysis, Concentrated reactions. \\
	\noindent 2010 \emph{Mathematics Subject Classification.} 35B25, 35B40, 35J92.

	\section{Introduction}
	
	We analyze the asymptotic behavior of solutions of a quasilinear equation with homogeneous Neumann boundary conditions set in a family of rough domains $\Omega^\varepsilon \subset \mathbb{R}^{n+1}$ with a nonlinear term concentrated on a neighborhood $\mathcal{O}^\varepsilon \subset \Omega^\varepsilon$ of the boundary $\partial \Omega^\varepsilon$. 
	We consider
	\begin{equation}\label{OD}
	\Omega^\varepsilon=\left\lbrace (x,y)\in \mathbb{R}^{n+1} \, : \, x \in \omega, -1<y<\varepsilon \, \psi(x) \, g({x}/{\varepsilon})\right\rbrace,\,\,\,\,0<\varepsilon\ll 1,
	\end{equation}
	where $\omega \subset \mathbb{R}^n$ is a bounded smooth domain,
		
	($\mathbf{H_D}$) $g : \mathbb{R}^n \mapsto \mathbb{R}$ is a strictly positive function, Lipschitz, periodic in the unitary cube $[0,1]^n$ with $0<g_0 \leq g(x) \leq g_1$. 
	Also, we take a $\mathcal{C}^{\infty}$-function $\psi: \mathbb{R}^n \mapsto \mathbb{R}$, $0 \leq \psi \leq 1$, with compact support in $\omega$.  

	Concerning to the narrow strip $\mathcal{O}^\varepsilon$ we set 
	\begin{equation} \label{strip}
	\mathcal{O}^\varepsilon=\{ (x,y)\in \mathbb{R}^{n+1} \, : \, x \in \omega, \, \varepsilon [ \psi(x) g\left({x}/{\varepsilon}\right)-\varepsilon^\gamma h(x, {x}/{\varepsilon^\beta})] < y < \varepsilon \psi(x) g\left({x}/{\varepsilon}\right)\}
	\end{equation}
		
	($\mathbf{H_T}$) $\gamma>0$ and $\beta\geq0$ are constants, $h:\omega \times \mathbb{R}^n \mapsto \mathbb{R}$ is a $\mathcal{C}^1$-function, nonnegative, bounded with bounded derivatives. Also, we take $h(x, \cdot): \mathbb{R}^n \mapsto \mathbb{R}$ periodic in the unitary cube $[0,1]^n$ for all $x \in \omega$.
	
	Notice $\Omega^\varepsilon$ uniformly converge to the cylinder $\Omega = \omega \times (-1,0)$ and $\mathcal{O}^\varepsilon$ degenerate to the interval $(0,1)$. Moreover, we allow both to present oscillatory boundary due to the periodicity assumptions on functions $g$ and $h$ modeling roughness. 
	 Parameters $\beta$ and $\gamma$ set respectively the roughness order of $\mathcal{O}^\varepsilon$ and its Lesbegue measure.

	We analyze the solutions of the problem 
	\begin{equation}\label{variational_l}\int_{\Omega^\varepsilon} \{|\nabla u_\varepsilon|^{p-2}\nabla u_\varepsilon\nabla \varphi+|u_\varepsilon|^{p-2}u_\varepsilon\varphi \}dxdy=\dfrac{1}{\varepsilon^{\gamma+1}}\int_{\mathcal{O}^\varepsilon}f(u_\varepsilon) \varphi dxdy, \quad \varphi\in W^{1,p}(\Omega^\varepsilon)\end{equation}
which is the variational formulation of the quasilinear equation 
\begin{equation}\label{problem}
	\left\lbrace\begin{array}{lll}
	-\Delta_p u_\varepsilon+|u_\varepsilon|^{p-2}u_\varepsilon={1}/{\varepsilon^{\gamma+1}}\chi_{\mathcal{O}^\varepsilon}f(u^\varepsilon)\mbox{ in }\Omega^\varepsilon\\[0.3cm]
		|\nabla u_\varepsilon|^{p-2}{\partial_{\nu^\varepsilon} u_\varepsilon} =0\mbox{ on }\partial \Omega^\varepsilon
	\end{array}\right.
	\end{equation}
	where $\nu^\varepsilon$ denotes the unit outward normal to the boundary $\partial \Omega^\varepsilon$, $2\leq p<\infty$ and 
	$
	\Delta_p\cdot = \mbox{div }\left(|\nabla \cdot|^{p-2}\nabla\cdot\right)
	$
	is the $p$-Laplacian differential operator. Also, $\chi_{\mathcal{O}^\varepsilon}$ is the characteristic function of $\mathcal{O}^\varepsilon$ and 
		
	($\mathbf{H_f}$) $f: \mathbb{R} \mapsto \mathbb{R}$ is a bounded function with bounded derivatives.

	Notice we are in agreement with previous works as \cite{gleice, gleice2, arrieta} using the characteristic function $\chi_{\mathcal{O}^\varepsilon}$ and term $1/\varepsilon^{\gamma+1}$ to express concentration on $\mathcal{O}^\varepsilon$. Our main goal here is to improve \cite{GS, AA, Nogueira, CFP}  dealing with a $p$-Laplacian equation in a perturbed $n+1$-dimensional domain presenting roughness on boundary. 
		
	The limit equation to \eqref{problem} is the $p$-Laplacian problem with nonlinear Neumann boundary condition 
	\begin{equation}\label{limit}
	\left\lbrace\begin{array}{lll}
	-\Delta_p u+|u|^{p-2}u = 0\mbox{ in } \Omega \\[0.3cm]
	|\nabla u|^{p-2}{\partial_{\nu} u} = 0 \mbox{ on } \partial \Omega \setminus \Gamma  \\[0.3cm]
	|\nabla u|^{p-2}{\partial_{\nu} u} = \mu(x) \, f(u) \mbox{ on } \Gamma 
	\end{array}\right.
	\end{equation}
	where $\Omega$ is the cylinder $\omega \times (-1,0)$, $\Gamma=\omega\times\{0\}$ is the upper boundary of $\Omega$, and $\mu: \omega \mapsto \mathbb{R}$ is set by
	\begin{equation}\label{mu}
	\mu(x) = \int_{[0,1]^n} h(x, s) \, ds.
	\end{equation}
	
	The effect of the geometry of $\mathcal{O}^\varepsilon$ is captured by function $\mu$ which is the mean value of $h(x, \cdot)$ in $[0,1]^n$ for each $x \in \omega$. On the other hand, under prescribed conditions, the rough domain $\Omega^\varepsilon$ does not affect the limit problem in an effective way. We have the following result.

	\begin{theorem}  \label{main}
	Let $u_\varepsilon$ be a family of solutions to \eqref{problem}.
	
	Then, up to a subsequence, there exist $u \in W^{1,p}(\Omega)$ satisfying \eqref{limit} such that 
	$$
	\| u_\varepsilon - u \|_{W^{1,p}(\Omega)} \to 0 \quad \textrm{ as } \varepsilon \to 0.
	$$  
	\end{theorem}
	
	In order to accomplish our goal, we study concentrated integrals on narrow strips from $\mathbb{R}^{n+1}$ improving results from \cite{Nogueira} at Section \ref{conint}. A proof to Theorem \ref{main} is given at Section \ref{con}.

	\section{Concentrated Integrals} \label{conint}
	
	To achieve a better understanding of the behavior from the nonlinear concentrated term, we first analyze what it is called by \cite{gleice,arrieta,Nogueira} the concentrated (or concentrating) integral. Next, we introduce a map to set our nonlinear term.

\begin{proposition}\label{lema_int_conc}
	If we define $G_\varepsilon(x)=\varepsilon\psi(x)g(x/\varepsilon)$, for $\varepsilon_0>0$ sufficiently small there is a constant $C>0$, independent of $\varepsilon\in(0,\varepsilon_0)$ and $u_\varepsilon\in W^{1,p}(\Omega^\varepsilon)$, such that, for all $1-1/p<s\leq1$,
	\begin{equation}\label{1}
	\dfrac{1}{\varepsilon^{\gamma+1}}\int_{\mathcal{O}^\varepsilon}|u_\varepsilon|^q\leq C \|u_\varepsilon\|^q_{L^q(\omega;W^{s,p}(-1,G_\varepsilon(x)))}\quad  \forall q\geq 1; \quad \textrm{ and }
	\end{equation}
	\begin{equation}\label{2}
	\dfrac{1}{\varepsilon^{\gamma+1}}\int_{\mathcal{O}^\varepsilon}|u_\varepsilon|^p\leq C \|u_\varepsilon\|^q_{W^{1,p}(\Omega^\varepsilon)} \ \forall q\leq p.
	\end{equation}
\end{proposition}
\begin{proof}
	Notice that \eqref{1} follows by changing $H^s$ space by $W^{s,p}$, for $1-1/p<s<1$, in \cite[Theorem 4.1]{Nogueira}. Consequently, to prove \eqref{2} we have
	\begin{align*}
	\|u^\varepsilon\|^p&_{L^p(\omega;W^{s,p}(-1,G_\varepsilon(x)))} = \int_\omega \|u^\varepsilon(x,\cdot)\|_{W^{s,p}(-1,G_\varepsilon(x))}^pdx  \leq \int_\omega \|u^\varepsilon(x,\cdot)\|_{W^{1,p}(-1,G_\varepsilon(x))}^pdx \leq \|u\|^p_{W^{1,p}(\Omega^\varepsilon)}
	\end{align*}
	since $L^p(\omega;W^{1,p}(-1,G_\varepsilon(x)))\subset L^p(\omega;W^{s,p}(-1,G_\varepsilon(x)))$ with constant of inclusion independent of $\varepsilon$ analogously to \cite[Proposition 3.6]{AMA-thin}.
	In particular, since $p\geq 2$, we have $\frac{p}{q}>1$ for all $q< p$ and then, 
	\begin{align*}
	\|\varphi_\varepsilon\|&^q_{L^q(0,1;W^{1,p}(0,G_\varepsilon(x)))}  = \int_0^1\left(\int_0^{G_\varepsilon(x)}|\varphi_\varepsilon(x,y)|^pdy\right)^{\frac{q}{p}}dx + \int_0^1\left(\int_0^{G_\varepsilon(x)}|\partial_x \varphi_\varepsilon(x,y)|^pdy\right)^{\frac{q}{p}}dx \\ &  \leq g_1^{\frac{p-q}{p}}\left(\int_0^1\int_0^{G_\varepsilon(x)}|\varphi_\varepsilon(x,y)|^pdxdy\right)^{\frac{q}{p}} + g_1^{\frac{p-q}{p}}\left(\int_0^1\int_0^{G_\varepsilon(x)}|\partial_x \varphi_\varepsilon(x,y)|^pdxdy\right)^{\frac{q}{p}} \leq C\|\varphi_\varepsilon\|^q_{W^{1,p}(\Omega^\varepsilon)}
	\end{align*}		
\end{proof}	
	
	Given $\varepsilon>0$, consider $G_\varepsilon(x)=\varepsilon\psi(x)g(x/\varepsilon)$ and, for $1-1/p<s<1$, define
	\begin{align}\label{def_f}
	F_\varepsilon : W^{1,p}(\Omega^\varepsilon)&\to (L^p(\omega;W^{s,p}(-1,G_\varepsilon(x))))' \nonumber \\
	u_\varepsilon & \mapsto \langle F_\varepsilon(u_\varepsilon),\varphi_\varepsilon\rangle=\dfrac{1}{\varepsilon^{\gamma+1}}\int_{\mathcal{O}^\varepsilon}f(u_\varepsilon)\varphi_\varepsilon, \quad \forall \varphi_\varepsilon\in L^p(\omega;W^{s,p}(-1,G_\varepsilon(x))).  
	\end{align}
	\begin{proposition}\label{f}
		If we call $X_\varepsilon=L^p(\omega;W^{s,p}(-1,G_\varepsilon(x)))$ and $X_\varepsilon'$ its dual space, the function $F_\varepsilon$ defined in \eqref{def_f} has the following properties:
		\begin{enumerate}[(a)]
			\item there exists $K>0$ independent of $\varepsilon$ such that $\sup_{u_\varepsilon\in W^{1,p}(\Omega^\varepsilon)}\|F_\varepsilon(u_\varepsilon)\|_{X_\varepsilon'}\leq K$
			\item $F_\varepsilon$ is a Lipschitz application uniformly in $\varepsilon$.
		\end{enumerate}
	\end{proposition}
	\begin{proof}
			$(a)$ For $u_\varepsilon\in X_\varepsilon$ we have
			$\|F_\varepsilon(u_\varepsilon)\|_{X_\varepsilon'}=\sup_{\|u_\varepsilon\|_{X_\varepsilon}=1}|\langle F_\varepsilon(u_\varepsilon),\varphi_\varepsilon \rangle|.$
			Then for $\varepsilon>0$, using $f$ is bounded and Proposition \ref{lema_int_conc},
			\begin{align*}
			\dfrac{1}{\varepsilon^{\gamma+1}}\int_{\mathcal{O}^\varepsilon}|f(u_\varepsilon)\varphi_\varepsilon|  \leq \left(\dfrac{1}{\varepsilon^{\gamma+1}}\int_{\mathcal{O}^\varepsilon}|f(u_\varepsilon)|^q\right)^{\frac{1}{q}}\left(\dfrac{1}{\varepsilon^{\gamma+1}}\int_{\mathcal{O}^\varepsilon}|\varphi_\varepsilon|^p\right)^{\frac{1}{p}} \leq C\sup_{x\in\mathbb{R}}|f(x)|h_1^{\frac{1}{q}}\|\varphi_\varepsilon\|_{X_\varepsilon}
			\end{align*}
			and then,  
			$$\|F_\varepsilon(u_\varepsilon)\|_{X_\varepsilon'}\leq \sup_{x\in\mathbb{R}}|f(x)|h_1^{1/q}\Rightarrow \sup_{u_\varepsilon\in W^{1,p}(\Omega^\varepsilon)}\|F_\varepsilon(u_\varepsilon)\|_{X_\varepsilon'}\leq K$$
			
			$(b)$ Using $f'$ bounded, by Proposition \ref{lema_int_conc}, we obtain, if $q\leq 2$ is the conjugate of $p\geq 2$,
			\begin{equation*}
			\begin{gathered}\dfrac{1}{\varepsilon^{\gamma+1}}\int_{\mathcal{O}^\varepsilon}|f(u_\varepsilon)\varphi_\varepsilon-f(v_\varepsilon)\varphi_\varepsilon|  =\dfrac{1}{\varepsilon^{\gamma+1}}\int_{\mathcal{O}^\varepsilon}|f(u_\varepsilon)-f(v_\varepsilon)||\varphi_\varepsilon| \\  \leq \left(\dfrac{1}{\varepsilon^{\gamma+1}}\int_{\mathcal{O}^\varepsilon}\left[\sup_{x\in\mathbb{R}}|f'(x)|^q\right]|u_\varepsilon - v_\varepsilon|^q\right)^{1/q}\left(\dfrac{1}{\varepsilon^{\gamma+1}}\int_{\mathcal{O}^\varepsilon}|\varphi_\varepsilon|^p\right)^{1/p} 
			\leq C
			\|u_\varepsilon-v_\varepsilon\|_{W^{1,p}(\Omega^\varepsilon)}|||\varphi_\varepsilon\|_{X_\varepsilon}
			\end{gathered}
			\end{equation*}
			Thus there exists $L>0$ such that
			$\|F_\varepsilon(u_\varepsilon)-F_\varepsilon(v_\varepsilon)\|_{X_\varepsilon'}\leq L\|u_\varepsilon-v_\varepsilon\|_{W^{1,p}(\Omega^\varepsilon)}.$
	\end{proof}

	\section{Convergence results} \label{con}
	
		Our next step is to analyze what happens when $\varepsilon$ goes to zero. 
		Since our problems are defined in varying domains $\Omega^\varepsilon$ that depends on $\varepsilon$, we will first set an extension operator which will help us compare functions in a fixed domain $U$ containing $\bar{\Omega}^\varepsilon$ for all $\varepsilon>0$.
		
		\begin{proposition}  \label{extensao}
			The family $\Omega^{\varepsilon}$ admits a continuous extension operator $P_\varepsilon:L^p(\Omega^{\varepsilon})\mapsto L^p(U)$, where the open set $U=U_1\times U_2 \subset \mathbb{R}^n\times\mathbb{R}$ is such that the closure of $\Omega^\varepsilon$ is contained in $U$ for all $\varepsilon>0$, and $\|P_\varepsilon u^\varepsilon\|_{W^{1,p}(U)}\leq C_0\|u^\varepsilon\|_{W^{1,p}(\Omega^{\varepsilon})}$, $\|P_\varepsilon u^\varepsilon\|_{L^p(U_1;W^{s,p}(U_2))}\leq C_s\|u^\varepsilon\|_{L^2(0,1;W^{s,p}(-1,G_\varepsilon(x)))}$ and $\|P_\varepsilon u^\varepsilon\|_{L^p(U)}\leq C_1\|u^\varepsilon\|_{L^p(\Omega^{\varepsilon})}$, 
			where $G_\varepsilon(x)=\varepsilon\psi(x)g(x/\varepsilon)$ and the constants $C_0,C_s,C_1>0$ are independent of $\varepsilon>0$ with $0\leq s\leq 1$.
		\end{proposition}
		\begin{proof}
			The proof follows with the extension operator defined in \cite[Proposition 2]{Nogueira}. 
		\end{proof}
	 
		Now, one can prove convergence results concerning to the concentrated integrals.
		
		\begin{proposition} \label{conv_produto_b}
			Let $U\subset\mathbb{R}^{n+1}$ an open set such that $\Omega^{\varepsilon}\subset U$ for all $\varepsilon>0$. Then,
			\begin{equation*}
			\dfrac{1}{\varepsilon^{\gamma+1}}\int_{\mathcal{O}^\varepsilon}u(x,y)\varphi(x,y)dydx \longrightarrow \int_\Gamma\mu u\varphi dS, \quad \textrm{ as } \varepsilon\to0,
			\end{equation*}
			for any $u,\varphi\in W^{1,p}(U)$ where $\mu(\cdot)$ is given by \eqref{mu} and $\Gamma=\omega\times\{0\}$.
		\end{proposition}
		
		\begin{proof}
			Due to \cite[Theorem 1.4.2.1]{grisvard}, we know that
			$C^\infty_c(\bar{U}):=\{u\in C^\infty(U); \ u=v_{|_{U}}, \text{ com } v\in C_c^\infty(\mathbb{R}^{n+1}) \}$
			is dense in $W^{1,p}(U)$ and we can assume $u,\varphi\in C^\infty_c(\bar{U})$. Consider $G_\varepsilon(x)=\varepsilon\psi(x)g(x/\varepsilon)$ and $H_\varepsilon(x)=h(x,x/\varepsilon^\beta)$. Then, performing the change of variables
			$$y_1=x, \quad y_2=(y-G_\varepsilon(x)+\varepsilon^{\gamma+1} H_\varepsilon(x))/({\varepsilon^{\gamma+1} H_\varepsilon(x)}),$$
			we get
			\begin{align*}
				&\dfrac{1}{\varepsilon^{\gamma+1}}\int_{\mathcal{O}^\varepsilon}u(x,y)\varphi(x,y)dydx=\dfrac{1}{\varepsilon^{\gamma+1}}\int_\omega\int_{G_\varepsilon(x)-\varepsilon^{\gamma+1} H_\varepsilon(x)}^{G_\varepsilon(x)}u(x,y)\varphi(x,y)dydx\\ =& \int_\omega\int_0^1 u(y_1,G_\varepsilon(y_1)-\varepsilon^{\gamma+1} H_\varepsilon(y_1)(1-y_2))\varphi(y_1,G_\varepsilon(y_1)-\varepsilon^{\gamma+1} H_\varepsilon(y_1)(1-y_2))H_\varepsilon(y_1)dy_2dy_1 \\ = & \int_\omega\int_0^1(u(y_1,G_\varepsilon(y_1)-\varepsilon^{\gamma+1} H_\varepsilon(y_1)(1-y_2))-u(y_1,0))
				\varphi(y_1,G_\varepsilon(y_1)-\varepsilon^{\gamma+1} H_\varepsilon(y_1)(1-y_2))H_\varepsilon(y_1)dy_2dy_1 \\ & + \int_\omega\int_0^1u(y_1,0)(\varphi(y_1,G_\varepsilon(y_1)-\varepsilon^{\gamma+1} H_\varepsilon(y_1)(1-y_2))-\varphi(y_1,0))H_\varepsilon(y_1)dy_2dy_1
				\\ &   + \int_\omega\int_0^1u(y_1,0)\varphi(y_1,0)(H_\varepsilon(y_1)-\mu(y_1))dy_2dy_1 
				+ \int_\omega\mu(y_1) u(y_1,0)\varphi(y_1,0)dy_1 
				\end{align*}
				When $\varepsilon\to0$, since $G_\varepsilon\to 0$ uniformly, we obtain 
				%
			\begin{equation*}
			\dfrac{1}{\varepsilon^{\gamma+1}}\int_{\mathcal{O}^\varepsilon}u(x,y)\varphi(x,y)dydx \to \int_\omega\mu(y_1)u(y_1,0)\varphi(y_1,0)dy_1 = \int_\Gamma\mu u \varphi dS.
			\end{equation*}
		\end{proof}
				
		We also have similar results concerning to nonlinearity $f$. The proof is analogous.
		\begin{corollary}\label{conv_produto2}
			Let $U\subset\mathbb{R}^{n+1}$ an open set such that $\Omega^{\varepsilon}\subset U$ for all $\varepsilon>0$. If $u,\varphi\in W^{1,p}(U)$ and $f : \mathbb{R}\to\mathbb{R}$ is a bounded function of class $C^1$, with bounded derivative, then
			\begin{align*}
			\dfrac{1}{\varepsilon^{\gamma+1}}\int_{\mathcal{O}^\varepsilon}f(u(x,y))\varphi(x,y)dydx \to \int_\Gamma\mu f(u) \varphi dS,
			\end{align*}
			as $\varepsilon\to0$, where $\mu(\cdot)$ is given by \eqref{mu}  and $\Gamma=\omega\times\{0\}$.
		\end{corollary}
%
	Next, we analyze convergence in concentrated integrals.
	
	\begin{proposition}\label{prop5.1}
			Let $w^\varepsilon$, $u^\varepsilon\in W^{1,p}(\Omega_\varepsilon)$ and $w,u\in W^{1,p}(U)$ such that $P_{\varepsilon} u^\varepsilon \rightharpoonup u$ and $P_{\varepsilon} w^\varepsilon \rightharpoonup w$ in $W^{1,p}(U)$ where $P_{\varepsilon}$ is the extension operator given by Proposition \ref{extensao}. Then
		\begin{equation*}
		\dfrac{1}{\varepsilon^{\gamma+1}}\int_{\mathcal{O}^\varepsilon}f(u^\varepsilon)w^\varepsilon \to \int_\Gamma\mu f(u) w dS,
		\end{equation*}
		where $\mu(\cdot)$ is given by \eqref{mu} and $\Gamma=\omega\times\{0\}$.
	\end{proposition}
	\begin{proof}
		Considering the extension operator $P_\varepsilon$ of Proposition \ref{extensao}, one gets by Corollary \ref{conv_produto2}
		\begin{align*}
		&\left|\dfrac{1}{\varepsilon^{\gamma+1}}\int_{\mathcal{O}^\varepsilon}f(u^\varepsilon)w^\varepsilon - \int_\Gamma\mu f(u)wdS\right| \leq
		\left|\dfrac{1}{\varepsilon^{\gamma+1}}\int_{\mathcal{O}^\varepsilon}f(u^\varepsilon)(w^\varepsilon-w)\right| \\ & \qquad  + \left|\dfrac{1}{\varepsilon^{\gamma+1}}\int_{\mathcal{O}^\varepsilon}(f(u^\varepsilon)-f(u))w\right|  + \left|\dfrac{1}{\varepsilon^{\gamma+1}}\int_{\mathcal{O}^\varepsilon}f(u)w - \int_\Gamma\mu f(u) w dS\right| \\
		\leq & \left( \dfrac{1}{\varepsilon^{\gamma+1}}\int_{\mathcal{O}^\varepsilon}|f(u^\varepsilon)|^q\right)^{\frac{1}{q}}\left( \dfrac{1}{\varepsilon^{\gamma+1}}\int_{\mathcal{O}^\varepsilon}|w^\varepsilon-w|^p\right)^{\frac{1}{p}} + \left( \dfrac{1}{\varepsilon^{\gamma+1}}\int_{\mathcal{O}^\varepsilon}|f(u^\varepsilon)-f(u)|^q\right)^{\frac{1}{q}}\left( \dfrac{1}{\varepsilon^{\gamma+1}}\int_{\mathcal{O}^\varepsilon}|w|^p\right)^{\frac{1}{p}} 
		\\& \qquad + \left|\dfrac{1}{\varepsilon^{\gamma+1}}\int_{\mathcal{O}^\varepsilon}f(u)w -\int_\Gamma\mu f(u) w dS\right| \\ \leq & \sup_{x\in\mathbb{R}}|f(x)|h_1^{\frac{1}{q}}\|w^\varepsilon-w\|_{X_\varepsilon}+\sup_{x\in\mathbb{R}}|f'(x)| \|u^\varepsilon-u\|_{X_\varepsilon}\|w\|_{W^{1,p}(\Omega^\varepsilon)} 
		+\left| \dfrac{1}{\varepsilon^{\gamma+1}}\int_{\mathcal{O}^\varepsilon}f(u)w-\int_\Gamma\mu f(u) w dS\right|  \\ & \leq K_1\|P_\varepsilon w^\varepsilon-w\|_{X_U}+K_2 \|P_\varepsilon u^\varepsilon-u\|_{X_U}\|w\|_{W^{1,p}(U)}+\left| \dfrac{1}{\varepsilon^{\gamma+1}}\int_{\mathcal{O}^\varepsilon}f(u)w-\int_\Gamma\mu f(u) w dS\right| \to 0.
		\end{align*}
	\end{proof}
	
	Finally, we show our main result.
	\begin{proof}[Proof of Theorem \ref{main}]
		
		\noindent\textit{(a) Uniform bound of solutions:}
		Take $\varphi=u_\varepsilon$ as a test function in \eqref{variational_l}. By Proposition \ref{f} (a), there exists $K>0$ such that
		$$
		\|u_\varepsilon\|_{W^{1,p}(\Omega^\varepsilon)}^p=\dfrac{1}{\varepsilon^{\gamma+1}}\int_{\mathcal{O}^\varepsilon}f(u_\varepsilon)u_\varepsilon dxdy\leq \sup_{v\in W^{1,p}(\Omega^\varepsilon)}\|F_\varepsilon(v)\|_{X_\varepsilon'}\leq K,
		$$
		which means that $u_\varepsilon$ is uniformly bounded in $W^{1,p}(\Omega^\varepsilon)$. 
		
		\noindent\textit{(b) Limiting problem:}
		Notice that we can rewrite our domain $\Omega^\varepsilon$ as $\Omega^\varepsilon=\mbox{int}(\overline{\Omega}\cup \overline{R^\varepsilon}),$
		where 
		$$
		R^\varepsilon=\left\lbrace (x,y)\in \mathbb{R}^{n+1} \, : \, x \in \omega, 0<y<\varepsilon \, \psi(x) \, g\left(\dfrac{x}{\varepsilon}\right)\right\rbrace
		$$
		and $\Omega=\omega\times (-1,0)$. Then, one can rewrite \eqref{variational_l} as
		\begin{equation}\label{variational_ldecomp}
		\begin{gathered}
		\int_{\Omega} \{|\nabla u_\varepsilon|^{p-2}\nabla u_\varepsilon\nabla \varphi+|u_\varepsilon|^{p-2}u_\varepsilon\varphi \}dxdy+\int_{R^\varepsilon} \{|\nabla u_\varepsilon|^{p-2}\nabla u_\varepsilon\nabla \varphi+|u_\varepsilon|^{p-2}u_\varepsilon\varphi \}dxdy\\=\dfrac{1}{\varepsilon^{\gamma+1}}\int_{\mathcal{O}^\varepsilon}f(u_\varepsilon) \varphi dxdy, \quad \varphi\in W^{1,p}(\Omega^\varepsilon).
		\end{gathered}
		\end{equation}
		
		Since $u_\varepsilon$ is uniformly bounded in $W^{1,p}(\Omega^\varepsilon)$, we have that $u_\varepsilon$ restricted to $\Omega$ is uniformly bounded in $W^{1,p}(\Omega)$. 
		Therefore, there is $u\in W^{1,p}(\Omega)$ such that, up to a subsequence,
		$$
		u_\varepsilon|_\Omega\rightharpoonup u\quad \mbox{ weakly in }\quad W^{1,p}(\Omega)\quad \mbox{ and strongly in }\quad L^p(\Omega).
		$$
		Notice that from the above convergence one can prove that  $|u_{\varepsilon}|^{p-2}u_{\varepsilon}\to |u|^{p-2}u$ strongly in $L^{p'}(\Omega)$. Furthermore, the integral over $R^\varepsilon$ converges to zero, because $R^\varepsilon$ is a thin domain and $u_\varepsilon$ is uniformly bounded. Moreover, by Proposition \ref{extensao}, there is $u^*\in W^{1,p}(U)$ such that, up to a subsequence, 
		\begin{equation}\label{Puepsilonconv}
		P_\varepsilon u_\varepsilon\rightharpoonup u^* \textrm{ weakly in }W^{1,p}(U)\textrm{ and strongly in }L^p(U).
		\end{equation}
		It is not difficult to see that 
		$
		u^*|_\Omega = u \textrm{ in }\Omega . 
		$ 
		
		Now, we identify the limit problem. For this sake, let $v\in W^{1,p}(\Omega)$. Using the monotonicity of $|\cdot|^{p-2}\cdot$ we have
		\begin{equation} \label{ineq}
		0\leq\int_{\Omega}(|\nabla u_\varepsilon|^{p-2}\nabla u_\varepsilon-|\nabla v|^{p-2}\nabla v)(\nabla u_\varepsilon-\nabla v).
		\end{equation} 
		Taking test functions as $\varphi=u_\varepsilon-P_{\varepsilon}v$, we can rewrite \eqref{variational_ldecomp} as
		\begin{equation*}
		\begin{gathered}
		\int_{\Omega}|\nabla u_\varepsilon|^{p-2}\nabla u_\varepsilon(\nabla u_\varepsilon-\nabla v) =\dfrac{1}{\varepsilon^{\gamma+1}}\int_{\mathcal{O}^\varepsilon}f(u_\varepsilon) (u_\varepsilon-P_\varepsilon v) dxdy\qquad\qquad\qquad \\ -\int_{\Omega} |u_\varepsilon|^{p-2}u_\varepsilon(u_\varepsilon-v) dxdy-\int_{R^\varepsilon} \{|\nabla u_\varepsilon|^{p-2}\nabla u_\varepsilon\nabla (u_\varepsilon-P_\varepsilon v)+|u_\varepsilon|^{p-2}u_\varepsilon(u_\varepsilon-P_\varepsilon v) \}dxdy.
		\end{gathered}
		\end{equation*}
		From \eqref{Puepsilonconv} and Proposition \ref{prop5.1}, we get the right hand side of the above expression converges to
		$$
		\int_{\Gamma} \mu f(u)(u-v)dS-\int_\Omega|u|^{p-2}u(u-v) dxdy.
		$$
		
		On the other hand,
		\begin{equation*}
		\int_{\Omega}|\nabla v|^{p-2}\nabla v(\nabla u_\varepsilon-\nabla v)\to \int_{\Omega}|\nabla v|^{p-2}\nabla v(\nabla u-\nabla v).
		\end{equation*}
		Thus,
		\begin{equation}\label{convaux}
		\begin{gathered}
		\int_{\Omega}(|\nabla u_\varepsilon|^{p-2}\nabla u_\varepsilon-|\nabla v|^{p-2}\nabla v)(\nabla u_\varepsilon-\nabla v)\\\to \int_{\Gamma}\mu f(u)(u-v)dS-\int_\Omega|u|^{p-2}u(u-v) dxdy-\int_{\Omega}|\nabla v|^{p-2}\nabla v(\nabla u-\nabla v),
		\end{gathered}
		\end{equation} 
		and by \eqref{ineq}, for any $v\in W^{1,p}(\Omega)$, 
		\begin{equation}\label{ineqhom}
		0\leq \int_{\Gamma}\mu f(u)(u-v)dS-\int_\Omega|u|^{p-2}u(u-v) dxdy-\int_{\Omega}|\nabla v|^{p-2}\nabla v(\nabla u-\nabla v).
		\end{equation}

		Now, let us take $v=u+\lambda \varphi$ for $\lambda>0$ and $\varphi\in W^{1,p}(\Omega)$ in \eqref{ineqhom}. If we divide it by $-\lambda$, we get
		$$
		0\geq \int_{\Gamma}\mu f(u)\varphi dS-\int_\Omega|u|^{p-2}u\varphi dxdy-\int_{\Omega}|\nabla u-\lambda\nabla\varphi|^{p-2}(\nabla u-\lambda\nabla\varphi)\nabla \varphi.
		$$
		Hence, if $\lambda\to0$ we get
		$$
		0\geq \int_{\Gamma}\mu f(u)\varphi dS-\int_\Omega|u|^{p-2}u\varphi dxdy-\int_{\Omega}|\nabla u|^{p-2}\nabla u\nabla \varphi.
		$$
		
		Proceeding in a similar way, we can get the reverse inequality for $v=u-\lambda \varphi$ getting the limit problem  
		$$
		0=\int_{\Gamma}\mu f(u)\varphi dS-\int_\Omega|u|^{p-2}u\varphi dxdy-\int_{\Omega}|\nabla u|^{p-2}\nabla u\nabla \varphi,\quad \forall\varphi\in W^{1.p}(\Omega).
		$$
		
		\noindent\textit{(c) Strong convergence:}
		Using the monotonicity of $|\cdot|^{p-2}\cdot$, there is a constant $c_{p}>0$ depending only on $p$ such that
		\begin{equation*}
		\|u_\varepsilon-u\|_{W^{1,p}(\Omega)}^p\leq c_p \int_{\Omega}(|\nabla u_\varepsilon|^{p-2}\nabla u_\varepsilon-|\nabla u|^{p-2}\nabla u)(\nabla u_\varepsilon-\nabla u).
		\end{equation*}
		By \eqref{convaux} with $v=u$, one gets  
		$
		\|u_\varepsilon-u\|_{W^{1,p}(\Omega)}\to 0.
		$
	\end{proof}

\end{document}